\documentclass[12pt]{article}
\usepackage{amssymb}
\usepackage{amssymb,amsthm}
\usepackage{amsmath}
\usepackage[english]{babel}
\usepackage{cite}
\usepackage{float}
\usepackage[]{graphicx}
\usepackage{sectsty}
\usepackage{caption}
\usepackage{subcaption}
\usepackage{lipsum}
\usepackage{hyperref}
\usepackage{xcolor}
\usepackage{lineno}
\hypersetup{
colorlinks = true,
linkcolor = blue,
citecolor = blue
}
\usepackage{hyperref}
\usepackage{xcolor}
\def\bc{\begin{center}}
\def\ec{\end{center}}

\def\s2c{\vskip 2cm}
\begin{document}
\renewcommand{\theequation}{\thesection.\arabic{equation}}
\numberwithin{equation}{section}
\def\mysection{\setcounter{equation}{0}\section}
\newtheorem{theorem}{Theorem}[section]
\newtheorem{problem}{Problem}[section]
\newtheorem{remark}[theorem]{remark}
\newtheorem{lemma}{Lemma}[section]
\theoremstyle{definition}
\newtheorem{example}[section]{Example}
\newtheorem{proposition}[theorem]{Proposition}
\newtheorem{Corollary}[theorem]{Corollary}
\newtheorem{definition}[theorem]{Definition}
\newtheorem{property}[]{Property}
\date{}

\title{\textbf{Faedo-Galerkin approximation technique to non-instantaneous impulsive abstract functional differential equations}}
\author {\textbf{Shahin Ansari and Muslim Malik } \\
School of Mathematical and Statistical Sciences,\\
Indian Institute of Technology Mandi, Kamand 175 005, India. \\
Emails: shahinansaripz@gmail.com, muslim@iitmandi.ac.in \\
  \\
   }
 \maketitle
  \author

\noindent \textbf{Abstract:} This manuscript is devoted to the study of a class of nonlinear non-instantaneous impulsive first order abstract retarded type functional differential equations in an arbitrary separable Hilbert space $\mathcal{H}$. A new set of sufficient conditions are derived to ensure the existence of approximate solutions. Finite dimensional approximations are derived using the projection operator. Through the utilization of analytic semigroup theory, fixed point theorem and Gronwall's inequality, we establish the uniqueness and convergence of approximate solutions. Additionally, we study the Faedo-Galerkin approximate solutions and establish some convergence results. Finally, an illustrative instance demonstrating the applications of obtained  results to partial differential equations is provided.
\section{Introduction}

An interesting and important area of study is the theory of differential equations in abstract spaces. By solving an abstract differential equation means that we are solving a class of differential equations, and hence we don't need to find solution of every differential equation from a given class of differential equations. From physical models, several problems in partial differential equations arise which can be reformulated as abstract differential equations in appropriate infinite dimensional space. This abstract formulation is made by suppressing the space variable and in this formulation the dimension of space is changed from finite to infinite. In abstract formulation, we utilize the powerful technique of semigroup of linear operators and functional analysis to establish the existence and uniqueness result of the abstract problems.\par 
Many real life problems can be modelled using initial value problems in differential equations, often exhibit memory effects. These effects can be better represented using more refined models with differential equations incorporate delayed or retarded arguments. Retarded functional differential equations constitute an extensive class of evolutionary equations. These equations are distinguished by the feature that, at any given point in time, the rate of change of the solution is contingent upon a discrete or distributed collection of previously achieved values of the solution. Hence, the initial value problem concern with retarded function differential equations is an infinite dimensional problem, which extends numerical and theoretical analysis beyond classical approaches tailored for functional element free differential equations.\par
In the real world, many evolving processes have the feature that at some points their states changes rapidly, such changes are referred to as  impulsive changes. In recent years, impulsive differential equations have become a natural framework for modelling of evolving processes and phenomena studied in population dynamics \cite{B1,B2}, ecology, control theory \cite{B3}, engineering \cite{B4} and so on. In the literature impulses can be categorized into two classes, namely instantaneous impulses (impulsive effect lasts for a very little period of time relative to the duration of the overall process) and non-instantaneous impulses (impulsive process remains active for a finite time period). For the details, related to impulsive problems, one can see \cite{B5,B6,B7,B8,3,8,S,9} and cited references therein. However, few dynamics of evolutionary processes in pharmacology could not be addressed by models with instantaneous impulses. For example, an insulin can be prescribed to a hyperglycemic patient and the injection of the medication into the blood circulatory system and the body’s consequential absorption are steady and ongoing processes that occur over a finite time span. This situation can be model by accounting a non-instantaneous impulse that begins with a jump and continuously proceeds for a finite duration. In \cite{B9}, Hernandez and O' Regan introduced the concept of non-instantaneous impulsive equations.\par
It has been noticed that many differential equations are hard to solve analytically and hence in this case we need to established the approximate solution. We have several approaches to the solvability of evolution problems in the time domain. The Faedo-Galerkin method is one of the most accomplished technique in which the mild solution of the proposed problem is the limit of the Cauchy sequence formed by approximate solutions. This method is extensively covered in the Lions's book \cite{F1}.\par
In 1978, Bazley \cite{F3,F4} utilized Heinz and Wahl's \cite{F2} existence result to study the following semilinear wave equation
\begin{align}\label{1.1}
\nu^{\prime \prime}(\varrho) + \mathcal{A} \nu(\varrho) &= N(\nu(\varrho)), \quad \varrho \geq 0, \nonumber \\
\nu(0) &= \phi, \quad \nu^{\prime}(0) = \theta,
\end{align}
and established the convergence of Faedo-Galerkin approximation of solutions to equation~(\ref{1.1}). After that, Goethel \cite{F5} extended this outcome to the following functional Cauchy problem
\begin{align}\label{1.2}
\nu^{\prime}(\varrho) + \mathcal{A} \nu(\varrho) &= N(\nu(\varrho)), \quad \varrho \geq 0, \nonumber \\
\nu(0) &= \phi,
\end{align}
by presuming $N$ is defined throughout the entire Hilbert space $\mathcal{H}$. Expanding upon Bazley's framework, Miletta \cite{F6} extended the work of Goethel \cite{F5} and proved the convergence of approximate solutions to (\ref{1.2}) by removing the constraints on $N$ as previously stipulated by Goethel \cite{F5}. Afterwards, a significant number of study has been done in this field. Muslim and Agarwal \cite{F7} have worked on the Faedo-Galerkin approximation of solutions to impulsive functional differential equations. Later on, some authors reported the work on approximation of solutions for fractional differential equations with impulsive conditions. For more study, one can see \cite{S2,S3,F8,F9}.
From the literature survey, we conclude that several good manuscripts on the Faedo-Galerkin approximation of solutions to various kind of differential equations have been published in many reputed journals. In contrast, this technique for non-instantaneous impulsive retarded differential equations is still an untreated topic. Thus, motivated by the above observation in this manuscript, we studied the  existence and convergence of Faedo-Galerkin approximate solutions of the following functional differential equation in a separable Hilbert space $\mathcal{H}$ with non-instantaneous impulses
\begin{align}\label{main1}
y^{\prime}(\varrho) &= -\mathcal{A}y(\varrho)+ \mathfrak{F}(\varrho, y_{\varrho}), \quad \varrho \in(\varsigma_i, \varrho_{i+1}], \ i = 0,1,2,\cdots,q, \nonumber \\
y(\varrho) &= h_i(\varrho, y(\varrho_i^-)), \quad \varrho \in (\varrho_i, \varsigma_i], \ i = 1,2,\cdots,q,  \\
y(\varrho) &= \chi(\varrho), \quad \varrho \in [-\tau,0], \quad \tau > 0 \nonumber
\end{align}
where $-\mathcal{A} : D(\mathcal{A}) \subset \mathcal{H} \rightarrow \mathcal{H}$ is the infinitesimal generator of an analytic semigroup of bounded linear operators $\mathfrak{T}(\varrho)_{\varrho \geq 0}$ defined on a separable Hilbert space $(\mathcal{H}, || . ||)$, pre-fixed numbers $\varsigma_i, \varrho_i$ fulfil the relation $0 = \varsigma_0 = \varrho_0 < \varrho_1 < \varsigma_1 < \varrho_2 < \cdots < \varsigma_i < \varrho_{i+1} = T$. For every $\varrho \geq 0$, $y_{\varrho} : [-\tau,0] \rightarrow \mathcal{H}$ is defined by $y_{\varrho}(\theta) = y(\varrho+\theta), \theta \in [-\tau,0]$ represents the history of $y$ at $\varrho$, $\chi \in \mathfrak{PC}_{0}$ is a given function, the maps $\mathfrak{F} \in C([0,T] \times \mathfrak{PC}_{0}; \mathcal{H})$, $h_i \in C((\varrho_i, \varsigma_i] \times \mathcal{H}; \mathcal{H})$ for all $i = 1,2, \cdots,q$ and $y(\varrho_{i}^-) = \lim_{ k \rightarrow 0^+}y(\varrho_{i}-k)$ stands for the left limit of $y(\varrho)$ at $\varrho = \varrho_i$.\\
The following is the format of the paper: In section \ref{Section 2}, we given some general assumptions and preliminary results. In section \ref{Section 3}, we study the existence and uniqueness of the approximate solutions for problem~(\ref{main1}). In section \ref{Section 4},  convergence of approximate solutions to the solution of the corresponding integral equation is established. Section \ref{Section 5} introduces Faedo-Galerkin approximate solutions for problem~(\ref{main1}) and show their convergence in the finite dimensional space. At last, in section \ref{Section 6}, an example is presented to show the validity of the obtained abstract results.

\section{Preliminaries}\label{Section 2}
This part introduces important results and certain key notations that will be employed throughout the residual of this manuscript. Let $M$ is a positive constant and $0 \in \rho(-\mathcal{A})$ such that $\parallel \mathfrak{T}(\varrho) \parallel \leq M $ for all $\varrho \geq 0$. Which allows us to define the $\alpha$-fractional power $\mathcal{A}^{\alpha} : D(\mathcal{A}^{\alpha}) \subset \mathcal{H} \rightarrow \mathcal{H}$ of $\mathcal{A}$ for $\alpha \in (0,1)$. $\mathcal{H}_{\alpha}$ represents $D(\mathcal{A}^{\alpha})$ forms a Banach space with the norm $\lVert y \rVert_{\alpha} = \ \rVert \mathcal{A}^{\alpha}y \lVert$. We encourage the readers to \cite{F13,F14,F15} for more study related to semigroup theory. \par
To address the impulses, we take into account the space of piecewise continuous functions $\mathfrak{PC}_{\varrho}^{\alpha} = \mathfrak{PC}([-\tau, \varrho];\mathcal{H}_{\alpha}) = \{y : [-\tau, \varrho] \rightarrow \mathcal{H}_{\alpha} : y \in C((\varrho_{i},\varrho_{i+1}];\mathcal{H}_{\alpha}), \ i = 0,1,2,\cdots,q \ \text{with} \ y(\varrho_{i}^{-}) = y(\varrho_{i})\  \text{and} \  y(\varrho_{i}^{+}) \  \text{exists}, \ i = 1,2,\cdots,q \}$. $\mathfrak{PC}_{\varrho}^{\alpha}$ forms a Banach space with the norm $\parallel y \parallel_{\varrho,\alpha} = \sup\limits_{t \in [-\tau, \varrho]} \lVert y(t) \rVert_{\alpha}$. In addition, $B_{R}(\mathcal{H},z)$ represents the closed ball with center at z and radius R.\\ 
Our findings can be demonstrated based on the following set of requisite assumptions:
\vspace{2mm}
\\
\textbf{(A1):} $\mathcal{A} : D(\mathcal{A}) \rightarrow \mathcal{H}$ such that $\overline{D(\mathcal{A})} = \mathcal{H}$ is positive definite, self-adjoint, closed linear operator and let $\mathcal{A}$ has pure point spectrum\\
\begin{equation*}
0 < \lambda_{0} \leq \lambda_{1} \leq \lambda_{2} \leq \cdots \leq \lambda_{n} \leq \cdots
\end{equation*}
where $\lambda_{n} \rightarrow \infty$ as $n \rightarrow \infty$ and an orthonormal system of eigenfunctions $\psi_{j}$ is complete, which  is corresponding to $\lambda_j$ i.e.
\begin{equation*}
\mathcal{A} \psi_{j} = \lambda_{j}\psi_{j} \quad \text{and} \quad \langle  \psi_{i} , \psi_{j} \rangle = \delta_{ij},
\end{equation*}
 where $\delta_{ij}$ denotes the Kronecker Delta function.
\vspace{2mm}
\\
\textbf{(A2):} The function ${\chi}(\varrho) \in D(\mathcal{A}^{\alpha})$ and is locally H$\ddot{\text{o}}$lder continuous on $[-\tau, 0]$. We define
\begin{equation*}
\bar{\chi}(\varrho) = \begin{cases}
 \chi(\varrho), \quad \varrho \in [-\tau,0],\\
 \chi(0), \quad \varrho \in [0,T].
\end{cases}
\end{equation*}
\vspace{2mm}
\\ 
\textbf{(A3):} The function $\mathfrak{F}: [0,T] \times \mathfrak{PC}_{0}^{\alpha} \rightarrow \mathcal{H}$ is continuous and satisfies
\begin{equation*}
\parallel \mathfrak{F}(\varrho,x)- \mathfrak{F}(\varsigma,y) \parallel \ \leq L_{\mathfrak{F}}(|\varrho-\varsigma|^{\eta}+ \parallel x-y \parallel_{0,\alpha})
\end{equation*}
for all  $\varrho,\varsigma \in [0,T]$, $\eta \in (0,1]$ and $y,x \in B_R(\mathfrak{PC}_{0}^{\alpha},\bar{\chi})$, where $L_{\mathfrak{F}}$ is a positive constant.
\vspace{2mm}
\\
\textbf{(A4):} The functions $h_{i}:(\varrho_{i},\varsigma_{i}] \times \mathcal{H}_{\alpha} \rightarrow \mathcal{H}_{\alpha}$, $i = 1,2,\cdots,q$ are continuous and there exist positive constants $L_{h_i}$, $i = 1,2,\cdots,q$ such that
\begin{equation*}
\parallel h_i(\varrho,x)- h_i(\varsigma,y) \parallel_{\alpha} \ \leq L_{h_i} [|\varrho-\varsigma| + \parallel x-y \parallel_{\alpha}]
\end{equation*}
also
\begin{equation*}
\parallel h_{i}(\varrho,x) \parallel_{\alpha} \ \leq M_{h}, \quad \forall \ x \in \mathcal{H}_{\alpha}, \quad \varrho \in (\varrho_{i},\varsigma_{i}],
\end{equation*}
for all $\varrho,\varsigma \in (\varrho_{i},\varsigma_{i}]$ and $y,x \in B_R(\mathcal{H}_{\alpha},\bar{\chi})$, where $M_{h}$ is a positive constant.\vspace{2mm}
\\
To simplify notation, we set \\
$B_1 = \parallel f(0,0) \parallel ; \quad L(R) = L_{\mathfrak{F}}(T^{\eta} + R + \parallel \bar{\chi} \parallel_{T,\alpha} )+ B_{1}; \quad Q_{i} = M_{\alpha} L_{\mathfrak{F}} \frac{\varrho_{i+1}^{1-\alpha}}{1-\alpha}, \quad i = 0,1,2,\cdots,q$;\\
$N_{0} = \sup\limits_{\varrho \in [0,T]} \parallel (\mathfrak{T}(\varrho)-I) \chi(0) \parallel_{\alpha}+ M_{\alpha}\big[L_{\mathfrak{F}}(T^{\eta} + \parallel \bar{\chi} \parallel_{T,\alpha})+ B_{1} \big] \frac{\varrho_{1}^{1-\alpha}}{1-\alpha}$;\\
$N_{i} = MM_{h}+ \lVert \bar{\chi}\rVert_{T,\alpha}+ M_{\alpha}\big[ L_{\mathfrak{F}}(T^{\eta} + \parallel \bar{\chi} \parallel_{T,\alpha})+ B_{1} \big] \frac{\varrho_{i+1}^{1-\alpha}}{1-\alpha}, \ i = 1,2,\cdots,q;$\\
$E_{0} = M_{\alpha}L_{\mathfrak{F}}\frac{\varrho_{1}^{1-\alpha}}{1-\alpha}; \quad E_{i} = M L_{h_{i}}+ M_{\alpha}L_{\mathfrak{F}}\frac{\varrho_{i+1}^{1-\alpha}}{1-\alpha}, \ i = 1,2,\cdots,q; \quad D = \max\limits_{1 \leq i \leq q}\bigg\{\max\limits_{0 \leq i \leq q} E_{i}, L_{h_i}\bigg\}.$

\begin{definition}
A function $y \in \mathfrak{PC}_{T}^{\alpha}$ is considered to be a mild solution of the problem (\ref{main1}) if $y(\varrho)$ fulfils the following integral equations:
\begin{equation}\label{2.4}
y(\varrho)=\begin{cases}
\ \chi(\varrho), \qquad \varrho \in [-\tau,0], \\
\ \mathfrak{T}(\varrho)\chi(0) + \int_{0}^{\varrho} \mathfrak{T}(\varrho-\varsigma)\mathfrak{F}(\varsigma,y_{\varsigma})d\varsigma, \quad \varrho \in [0,\varrho_{1}], \\
\ h_{i}(\varrho,y(\varrho_{i}^{-})), \quad \varrho \in (\varrho_{i},\varsigma_{i}], \quad i=1,2,\cdots,q, \\
\ \mathfrak{T}(\varrho-\varsigma_{i})h_{i}(\varsigma_{i},y(\varrho_{i}^{-}))+ \int_{\varsigma_{i}}^{\varrho} \mathfrak{T}(\varrho-\varsigma)\mathfrak{F}(\varsigma,y_{\varsigma})d\varsigma, \quad \varrho \in (\varsigma_{i},\varrho_{i+1}],\\
\ \qquad \qquad \qquad \qquad \qquad \qquad \qquad \qquad \qquad \quad i = 1,2,\cdots,q.
\end{cases}
\end{equation}	
\end{definition}
\section{{\large\textbf{Existence of approximate solutions}}}\label{Section 3}
In this section, the main existence and convergence results for the system (\ref{main1}) are established. The following results are required to demonstrate the existence of approximate solutions to the problem.\\
From our assumptions on the operator $\mathcal{A}$, we infer that \\
\begin{equation}\label{3.1}
\parallel \mathcal{A}^{\alpha} \mathfrak{T}(\varrho)\parallel \leq M_{\alpha}\varrho^{-\alpha}, \quad \varrho \in (0,\infty),
\end{equation}
where $M_{\alpha}$ is a positive constant. Also
\begin{equation}\label{3.2}
\parallel (\mathfrak{T}(\xi)-I)x \parallel \leq M_{\nu}^{\prime} \xi^{\nu} \parallel x \parallel_{\nu}, \quad \xi > 0,
\end{equation}
where $ x \in D(\mathcal{A}^{v})$ and $0 < v \leq 1$. Let $v$ be any real number such that $\{1-\alpha \} > v > 0$, then for any $ u \in D(\mathcal{A}^{\alpha+v})$, we have $\mathcal{A}^{\alpha}u \in D(\mathcal{A}^{v})$.
Now for any $\varsigma, \varrho \in (0,T],\ \varsigma < \varrho$ and $\delta > 0$, we have the following inequalities
\begin{align}
\label{3.3} \parallel (\mathfrak{T}(\delta) -I)\mathcal{A}^{\alpha}\mathfrak{T}(\varrho-\varsigma)\parallel &\leq  M_{\nu}^{\prime} M_{\alpha+\nu}\delta^{v}(\varrho-\varsigma)^{-(\alpha+v)},\\
\label{3.4} \parallel (\mathfrak{T}(\delta) -I)\mathcal{A}^{\alpha}\mathfrak{T}(\varrho)\parallel &\leq \frac{M_{\nu}^{\prime} M_{\alpha+\nu} \delta^{v}}{\varrho^{\alpha+v}}.
\end{align}
Now, we define $\mathcal{H}_n = \text{span}\{\psi_0, \psi_1, \psi_2 \cdots \psi_n\}\subseteq \mathcal{H}$. Clearly, $\mathcal{H}_n$ is a finite dimensional subspace of $\mathcal{H}$
and let $\mathcal{P}^n: \mathcal{H} \rightarrow \mathcal{H}_n$ be the corresponding projection operators for each $n \in \mathbb{N}_{0}.$\par
We define
\begin{align*}
\mathfrak{F}_{n}: [0,T] \times \mathfrak{PC}_{0}^{\alpha} \rightarrow \mathcal{H}
\end{align*}
such that
\begin{align*}
\mathfrak{F}_{n}(\varrho,x) = \mathfrak{F}(\varrho,\mathcal{P}^{n}x)
\end{align*}
and
\begin{equation*}
h_{\text{i,n}} : (\varrho_{i}, \varsigma_i] \times \mathcal{H}_{\alpha} \rightarrow \mathcal{H}_{\alpha}
\end{equation*}
such that
\begin{equation*}
h_{\text{i,n}}(\varrho,x) = h_{i}(\varrho,\mathcal{P}^{n}x).
\end{equation*}
Consider a subset $B_{R}(\mathfrak{PC}_{T}^{\alpha},\bar{\chi}) \subseteq \mathfrak{PC}_{T}^{\alpha}$ defined as\\
\[B_R(\mathfrak{PC}_{T}^{\alpha},\bar{\chi}) = \{ z \in \mathfrak{PC}_{T}^{\alpha} : \lVert z-\bar{\chi} \rVert_{T,\alpha} \ \leq R \},\]
where
\begin{equation*}
R = \text{max}\bigg( \max\limits_{0 \leq i \leq q}\frac{N_{i}}{1-Q_{i}},\ M_{h}+ \lVert \bar{\chi} \rVert_{T,\alpha}\bigg).
\end{equation*}
We define the maps $\phi_{n}$ on $B_{R}(\mathfrak{PC}_{T}^{\alpha},\bar{\chi})$, $n \in \mathbb{N}_{0},$ as follows,
\begin{equation}
(\phi_{n}y)(\varrho)=\begin{cases}
\ \chi(\varrho), \qquad \varrho \in [-\tau,0], \\
\ \mathfrak{T}(\varrho)\chi(0) + \int_{0}^{\varrho} \mathfrak{T}(\varrho-\varsigma)\mathfrak{F}_{n}(\varsigma,y_{\varsigma})d\varsigma, \quad \varrho \in [0,\varrho_{1}],\nonumber \\
\ h_{i,n}(\varrho,y(\varrho_{i}^{-})), \quad \varrho \in (\varrho_{i},\varsigma_{i}], \quad i=1,2,\cdots,q, \nonumber \\
\ \mathfrak{T}(\varrho-\varsigma_{i})h_{i,n}(\varsigma_{i},y(\varrho_{i}^{-}))+ \int_{\varsigma_{i}}^{\varrho} \mathfrak{T}(\varrho-\varsigma)\mathfrak{F}_{n}(\varsigma,y_{\varsigma})d\varsigma, \quad \varrho \in (\varsigma_{i},\varrho_{i+1}],\\
\ \qquad \qquad \qquad \qquad \qquad \qquad \quad  i = 1,2,\cdots,q,
\end{cases}
\end{equation}
for any $y \in B_{R}(\mathfrak{PC}_{T}^{\alpha},\bar{\chi})$.
\vspace{2mm}
\\
To prove further results, we use the inequalities (\ref{3.1})-(\ref{3.4}).
\begin{theorem}\label{thrm1}
Let $\chi(\varrho) \in D(\mathcal{A})\  \forall \ \varrho \in [-\tau,0]$ and conditions $\textbf{(A1)}-\textbf{(A4)}$ along with $D < 1$ are hold. Then there exists a unique $y_{n} \in B_{R}(\mathfrak{PC}_{T}^{\alpha}, \bar{\chi})$ such that $\phi_{n}y_{n} = y_{n}$ for each $n \in \mathbb{N}_{0}$ and $y_n$ satisfies the following approximate solution
\begin{eqnarray}\label{main2}
y_{n}(\varrho)=
\begin{cases}
\ \chi(\varrho), \qquad \varrho \in [-\tau,0], \\
\ \mathfrak{T}(\varrho)\chi(0) + \int_{0}^{\varrho} \mathfrak{T}(\varrho-\varsigma)\mathfrak{F}_{n}(\varsigma,(y_{n})_{\varsigma})d\varsigma, \quad \varrho \in [0,\varrho_{1}], \\
\ h_{i,n}(\varrho,y_n(\varrho_{i}^{-})), \quad \varrho \in (\varrho_{i},\varsigma_{i}], \quad i=1,2,\cdots,q, \\
\ \mathfrak{T}(\varrho-\varsigma_{i})h_{i,n}(\varsigma_{i},y_n(\varrho_{i}^{-}))+ \int_{\varsigma_{i}}^{\varrho} \mathfrak{T}(\varrho-\varsigma)\mathfrak{F}_{n}(\varsigma, (y_{n})_{\varsigma})d\varsigma, \ \varrho \in (\varsigma_{i},\varrho_{i+1}], \\
\ \qquad \qquad \qquad \qquad \qquad \qquad \qquad \qquad i = 1,2,\cdots,q.
\end{cases}
\end{eqnarray}
\end{theorem}
\begin{proof}
 The theorem's proof is split into three steps for ease of understanding.\\
\textbf{Step 1:} Initially, we demonstrate the continuity of $\phi_{n}$ from $B_{R}(\mathfrak{PC}_{T}^{\alpha},\bar{\chi})$ to $B_{R}(\mathfrak{PC}_{T}^{\alpha},\bar{\chi})$.\\
For any $y \in B_R(\mathfrak{PC}_T^{\alpha}, \bar{\chi})$ and $\mu_1,\mu_2 \in [-\tau,0]$ with $\mu_2 > \mu_1$, we have 
\begin{equation}\label{3.3}
\lVert (\phi_{n} y)(\mu_2)-(\phi_{n} y)(\mu_1) \rVert_{\alpha} = \lVert {\chi}(\mu_2)-{\chi}(\mu_1) \rVert_{\alpha}.
\end{equation}
 For any $y \in B_{R}(\mathfrak{PC}_{T}^{\alpha},\bar{\chi})$ and $\mu_1,\mu_2 \in (0,\varrho_{1}]$ with $\mu_2 > \mu_1$,
\begin{align}\label{3.5}
\parallel (\phi_{n}y)(\mu_2) - (\phi_{n}y)(\mu_1) \parallel_{\alpha} & \ \leq \parallel \big(\mathfrak{T}(\mu_2-\mu_1)-I\big)\mathcal{A}^{\alpha}\mathfrak{T}(\mu_1) \parallel \parallel \chi(0) \parallel  \nonumber \\
& \quad + \int_{0}^{\mu_1} \parallel \big(\mathfrak{T}(\mu_2-\mu_1)-I\big)\mathcal{A}^{\alpha} \mathfrak{T}(\mu_1-\varsigma) \parallel \parallel \mathfrak{F}_{n}(\varsigma,y_{\varsigma}) \parallel d\varsigma  \nonumber\\
&\quad + \int_{\mu_1}^{\mu_2} \parallel \mathfrak{T}(\mu_2-\varsigma) \mathcal{A}^{\alpha} \parallel \parallel \mathfrak{F}_{n}(\varsigma,y_{\varsigma}) \parallel d\varsigma \nonumber \\
&\leq M_{\nu}^{\prime} M_{\alpha+\nu}(\mu_2-\mu_1)^{\nu}(\mu_1)^{-(\alpha+\nu)}\parallel \chi(0) \parallel + M_{\nu}^{\prime} M_{\alpha+\nu}(\mu_2-\mu_1)^{\nu}\nonumber\\
& \quad \times \frac{{\varrho_{1}}^{1-(\alpha+\nu)}}{1-(\alpha+\nu)} L(R) +M_{\alpha}\frac{(\mu_2-\mu_1)^{1-\alpha}}{1-\alpha}L(R).
\end{align}
Now, for any $y \in B_{R}(\mathfrak{PC}_{T}^{\alpha},\bar{\chi})$ and $\mu_1,\mu_2 \in (\varrho_{i},\varsigma_{i}],\ i = 1,2,\cdots,q$ with $\mu_2 > \mu_1$,
\begin{align}\label{3.6}
\parallel (\phi_{n}y)(\mu_2) - (\phi_{n}y)(\mu_1) \parallel_{\alpha} & \ = \ \parallel h_{i,n}(\mu_2,y(\varrho_{i}^{-})) - h_{i,n}(\mu_1,y(\varrho_{i}^{-})) \parallel_{\alpha} \nonumber \\
& \leq L_{h_i}|\mu_2-\mu_1|.
\end{align}
Similarly, for any $y \in B_{R}(\mathfrak{PC}_{T}^{\alpha},\bar{\chi})$ and $\mu_1,\mu_2 \in (\varsigma_{i},\varrho_{i+1}],\ i = 1,2,\cdots,q$ with $\mu_1 < \mu_2$,
\begin{align}\label{3.7}
\parallel (\phi_{n}y)(\mu_2) - (\phi_{n}y)(\mu_1) \parallel_{\alpha} & \ \leq \ \parallel (\mathfrak{T}(\mu_2 - \mu_1)-I)\mathcal{A}^{\alpha}\mathfrak{T}(\mu_1-\varsigma_{i}) \parallel \parallel h_{i,n}(\varsigma_i,y(\varrho_i^{-}))\parallel \nonumber \\
& \quad + \int_{\varsigma_{i}}^{\mu_1} \parallel (\mathfrak{T}(\mu_2 - \mu_1)-I)\mathcal{A}^{\alpha}\mathfrak{T}(\mu_1-\varsigma) \parallel \parallel \mathfrak{F}_{n}(\varsigma,y_{\varsigma}) \parallel d\varsigma \nonumber \\
& \quad + \int_{\mu_1}^{\mu_2} \parallel \mathcal{A}^{\alpha}\mathfrak{T}(\mu_2-\varsigma)\parallel \parallel \mathfrak{F}_{n}(\varsigma,y_{\varsigma}) \parallel d\varsigma \nonumber \\
& \leq M_{\nu}^{\prime}M_{\alpha+\nu}(\mu_2-\mu_1)^{\nu}(\mu_1-\varsigma_{i})^{-(\alpha+\nu)} \parallel h_{i,n}(\varsigma_i,y(\varrho_i^{-}))\parallel  \nonumber \\
& \quad + M_{\nu}^{\prime}M_{\alpha+\nu} (\mu_2-\mu_1)^{\nu}\frac{(\mu_1-\varsigma_i)^{1-(\alpha+\nu)}}{1-(\alpha+\nu)} L(R)\nonumber \\
& \quad +M_{\alpha}\frac{(\mu_2-\mu_1)^{1-\alpha}}{1-\alpha} L(R).
\end{align}
In the uniform operator topology, continuity of $\mathfrak{T}(\varrho)$ for $\varrho > 0$ implies that Eq.(\ref{3.3}) and inequalities (\ref{3.5})-(\ref{3.7}) tends to zero as $ \mu_2 \rightarrow \mu_1$ and the map $\varrho \rightarrow (\phi_ny)(\varrho)$ is continuous as a result.
\vspace{2mm}
\\
\textbf{Step 2:} Now, we need to show that the map $\phi_{n}$ is well defined. To show this, it suffices to verify that
\begin{equation*}
\parallel (\phi_{n}y)(\varrho) - \bar{\chi}(\varrho) \parallel_{\alpha}  \leq  R.
\end{equation*}
 For any $\varrho \in [-\tau,0]$, we have
\begin{equation}\label{3.8}
 (\phi_n y)(\varrho) - \bar{\chi}(\varrho)  = 0.
\end{equation}
Now, for any $y \in B_{R}(\mathfrak{PC}_{T}^{\alpha},\bar{\chi})$ and $\varrho \in (0,\varrho_1],$ we have
\begin{align}\label{3.9}
\parallel (\phi_{n}y)(\varrho)-\bar{\chi}(\varrho) \parallel_{\alpha} & \ \leq \parallel (\mathfrak{T}(\varrho)-I) \chi(0) \parallel_{\alpha} + \int_{0}^{\varrho} \parallel \mathcal{A}^{\alpha}\mathfrak{T}(\varrho-\varsigma) \parallel\parallel \mathfrak{F}_{n}(\varsigma,y_{\varsigma}) \parallel d\varsigma \nonumber \\
& \leq \parallel (\mathfrak{T}(\varrho)-I) \chi(0) \parallel_{\alpha} + \ M_{\alpha}L_{\mathfrak{F}}\big(T^{\eta}+R+\parallel \bar{\chi} \parallel_{T,\alpha}\big) \frac{\varrho^{1-\alpha}}{1-\alpha} \nonumber \\
&\quad + M_{\alpha} B_{1}\frac{\varrho^{1-\alpha}}{1-\alpha} \nonumber \\
& \leq N_{0} + Q_{0} R \quad \leq R.
\end{align}
Similarly, for any $y \in B_{R}(\mathfrak{PC}_{T}^{\alpha},\bar{\chi})$ and $ \varrho \in (\varrho_{i},\varsigma_{i}], i = 1,2,\cdots,q$,
\begin{align}\label{3.10}
\parallel (\phi_{n}y)(\varrho) - \bar{\chi}(\varrho) \parallel_{\alpha} & \ \leq \ M_{h} + \parallel \bar{\chi} \parallel_{T,\alpha} \quad \leq R .
\end{align}
Similarly, for any $y \in B_{R}(\mathfrak{PC}_{T}^{\alpha},\bar{\chi})$ and $\varrho \in (\varsigma_i, \varrho_{i+1}], i = 1,2,\cdots,q,$
\begin{align}\label{3.11}
\parallel (\phi_{n}y)(\varrho) - \bar{\chi}(\varrho) \parallel_{\alpha} & \leq \parallel \bar{\chi} \parallel_{T,\alpha} + \parallel \mathfrak{T}(\varrho-\varsigma_i) \parallel \parallel \mathcal{A}^{\alpha} h_{i,n}(\varsigma_i, y(\varrho_{i}^{-})) \parallel \nonumber \\
& \quad + \int_{\varsigma_{i}}^\varrho \parallel \mathcal{A}^{\alpha} \mathfrak{T}(\varrho-\varsigma) \parallel \parallel \mathfrak{F}_{n}(\varsigma,y_{\varsigma}) \parallel d\varsigma \nonumber \\
& \leq \parallel \bar{\chi} \parallel_{T,\alpha} + MM_h + M_{\alpha}L_{\mathfrak{F}}\big(T^{\eta}+R+\parallel \bar{\chi} \parallel_{T,\alpha}\big) \frac{(\varrho-\varsigma_i)^{1-\alpha}}{1-\alpha} \nonumber \\
&\quad + M_{\alpha} B_{1}\frac{(\varrho-\varsigma_i)^{1-\alpha}}{1-\alpha} \nonumber \\
& \leq N_i + Q_{i} R \quad \leq R.
\end{align}
After summarizing the above inequalities (\ref{3.9}), (\ref{3.10}) and (\ref{3.11}), we conclude that
\begin{equation*}
\parallel \phi_{n}y  - \bar{\chi} \parallel_{T,\alpha} \  \leq R, \quad \text{for all}\quad \varrho \in [-\tau,T].
\end{equation*}
Hence, $\phi_n$ maps $B_R(\mathfrak{PC}_{T}^{\alpha},\bar{\chi})$ into $B_R(\mathfrak{PC}_{T}^{\alpha},\bar{\chi})$.
\vspace{3mm}
\\
\textbf{Step 3:} Here, we demonstrate that $\phi_n$ is a strict contracting map. For any $y,x \in B_R(\mathfrak{PC}_T^{\alpha}, \bar{\chi})$ and $\varrho \in [-\tau,0],$
\begin{align}\label{3.12}
\parallel (\phi_{n} y)(\varrho) - (\phi_{n} x)(\varrho) \parallel_{\alpha} = 0.
\end{align} 
Now, for any $y,x \in B_{R}(\mathfrak{PC}_{T}^{\alpha},\bar{\chi})$ and $\varrho \in (0,\varrho_1]$,
\begin{align}\label{3.13}
\parallel (\phi_{n}y)(\varrho)-(\phi_{n}x)(\varrho) \parallel_{\alpha} &\leq \int_0^\varrho \parallel \mathcal{A}^{\alpha}\mathfrak{T}(\varrho-\varsigma) \parallel \mathfrak{F}_n(\varsigma,y_{\varsigma})- \mathfrak{F}_n(\varsigma,x_{\varsigma}) \parallel d\varsigma \nonumber \\
& \leq E_{0} \parallel y-x \parallel_{T,\alpha}.
\end{align}
Also, for any $y,x \in B_{R}(\mathfrak{PC}_{T}^{\alpha},\bar{\chi})$ and $\varrho \in (\varrho_i, \varsigma_i], \ i = 1,2,\cdots,q$,
\begin{align}\label{3.14}
\parallel (\phi_{n}y)(\varrho)-(\phi_{n}x)(\varrho) \parallel_{\alpha} &= \ \parallel h_{i,n}(\varrho, y(\varrho_i^{-})) - h_{i,n}(\varrho,x(\varrho_i^{-})) \parallel_{\alpha} \nonumber \\
& \leq L_{h_i} \parallel y-x \parallel_{T,\alpha}.
\end{align}
Similarly, for any $y,x \in B_{R}(\mathfrak{PC}_{T}^{\alpha},\bar{\chi})$ and $\varrho \in (\varsigma_i, \varrho_{i+1}], \ i = 1,2,\cdots,q,$
\begin{align}\label{3.15}
\parallel (\phi_ny)(\varrho)-(\phi_nx)(\varrho) \parallel_{\alpha} & \leq \parallel \mathfrak{T}(\varrho-\varsigma_i) \parallel \parallel h_{i,n}(\varsigma_i,y(\varrho_i^{-})) - h_{i,n}(\varsigma_i,x(\varrho_i^{-})) \parallel_{\alpha} \nonumber \\
&\quad + \int_{\varsigma_i}^{\varrho} \parallel \mathcal{A}^{\alpha} \mathfrak{T}(\varrho-\varsigma) \parallel \parallel \mathfrak{F}_n(\varsigma,y_{\varsigma}) - \mathfrak{F}_n(\varsigma,x_{\varsigma}) \parallel d\varsigma \nonumber \\
& \leq  M L_{h_i} \parallel y-x \parallel_{T,\alpha} + M_{\alpha} \frac{(\varrho-\varsigma_i)^{1-\alpha}}{1-\alpha} L_{\mathfrak{F}} \parallel y-x \parallel_{0,\alpha} \nonumber \\
& \leq E_{i} \parallel y-x \parallel_{T,\alpha}.
\end{align}
Therefore, by the Eq. (\ref{3.12}) and inequalities (\ref{3.13})-(\ref{3.15}), we get
\begin{equation*}
\parallel \phi_{n}y - \phi_{n}x \parallel_{T,\alpha} \ \leq \ D \parallel y-x \parallel_{T,\alpha}.
\end{equation*}
Therefore, we can conclude that $\phi_{n}$ is a strict contracting operator. As a result, Banach contraction principle implies that $\phi_{n}$ has a unique fixed point $y_n \in B_R(\mathfrak{PC}_{T}^{\alpha},\bar{\chi})$, such that $\phi_{n}y_{n} = y_{n}$, which satisfies the equation (\ref{main2}), as a result, the proof is completed.
\end{proof}
\begin{Corollary} Let $\textbf{(A1)}-\textbf{(A4)}$ are fulfilled, $\chi(\varrho) \in D(\mathcal{A}) \ \forall \ \varrho \in [-\tau,0]$ and $h_{i,n}(\varrho,y_{n}(\varrho_i^-)) \in D(\mathcal{A}) \ \forall$ $\varrho \in (\varrho_i, \varsigma_i], \ i = 1,2,\cdots,q$ then $y_n(\varrho) \in D(\mathcal{A}^{\beta}) \ \forall\varrho \in [-\tau,T]$, where $\beta \in [0, 1).$
\end{Corollary}
\begin{proof} This result is true for all $\varrho \in [-\tau,0]$ and $\varrho \in (\varrho_{i}, \varsigma_i]$. Now, we only need to prove this result for $\varrho \in (\varsigma_i, \varrho_{i+1}]$. The existence of a unique $y_{n} \in B_R(\mathfrak{PC}_{T}^{\alpha},\bar{\chi})$ is established by Theorem (\ref{thrm1}) such that $y_n$ satisfying the integral equation (\ref{main2}). By Theorem 2.6.13 in \cite{F13}, we have $\mathfrak{T}(\varrho) : \mathcal{H} \rightarrow D(\mathcal{A}^{\eta})$ for $\varrho > 0$ and $\eta \in [0,1),$ and $0 \leq \beta \leq \nu < 1$ implies that $ D(\mathcal{A}^{\nu}) \subseteq D(\mathcal{A}^{\beta})$. In Pazy \cite{F13} by Theorem 1.2.4, we obtain $ \int_0^\varrho \mathfrak{T}(\varsigma)x d\varsigma  \in D(\mathcal{A})$ for any $x \in \mathcal{H}$. Also, for $\eta \geq 0$ and for any $x \in \mathcal{H}$, $\mathfrak{T}(\varrho)x \in D(\mathcal{A}^{\eta})$. Hence, the required outcome derives from the previously mentioned facts, coupled with the result that $D(\mathcal{A}) \subseteq D(\mathcal{A}^{\beta})$ for $\beta \in [0, 1]$.
\end{proof}

\begin{lemma}
Let $\textbf{(A1)}-\textbf{(A4)}$ hold and $\chi(\varrho) \in D(\mathcal{A})$ $\forall \varrho \in [-\tau,0]$ then there exists a constant $M^{\prime}$, independent of n, such that
\begin{equation*}
\parallel \mathcal{A}^{\beta}y_n(\varrho) \parallel \ \leq \ M^{\prime}, \quad \forall \ \varrho \in [-\tau, T] \quad  \text{and} \  0 \leq \beta < 1, \quad \text{where}
\end{equation*}
\begin{equation*}
M^{\prime} = \max \{ \parallel \chi \parallel_{0,\beta}, \mathcal{G}, M_h \}, \quad \text{where} \quad \mathcal{G} = \max\limits_{0 \leq i \leq q} M_i.
\end{equation*}
\end{lemma}
\begin{proof} Operate $\mathcal{A}^{\beta}$ on both sides of Eq. (\ref{main2}), and for any $\varrho \in [-\tau,0]$, we have
\begin{align}
\parallel y_n(\varrho) \parallel_{\beta} \leq \parallel \chi(\varrho) \parallel_{\beta} \leq \parallel \chi \parallel_{0,\beta}, \qquad \forall \ \varrho \in [-\tau,0].
\end{align} 

For $\varrho \in (0, \varrho_1]$, applying $\mathcal{A}^{\beta}$ on both sides of Eq. (\ref{main2}), we get
\begin{align}
\parallel y_n(\varrho) \parallel_{\beta} & \leq M \parallel \chi \parallel_{0,\beta} + \int_{0}^{\varrho} \parallel \mathcal{A}^{\beta}\mathfrak{T}(\varrho-\varsigma) \parallel  \parallel \mathfrak{F}_n(\varsigma, (y_{n})_{\varsigma}) \parallel d\varsigma \nonumber \\
& \leq M \parallel \chi \parallel_{0,\beta} + M_{\beta} \frac{\varrho_1^{1-\beta}}{1-\beta}L(R) \nonumber \\
& = M_0 \ \leq \ M^{\prime}.
\end{align}
For $ \varrho \in (\varrho_i,\varsigma_i], \ i = 1,2,\cdots,q$, applying $\mathcal{A}^{\beta}$ on both sides of Eq. (\ref{main2}), we have
\begin{align}
\parallel y_n(\varrho) \parallel_{\beta} &\leq M_h.
\end{align}
Similarly for $ \varrho \in (\varsigma_i,\varrho_{i+1}], \ i = 1,2,\cdots,q,$ after applying $\mathcal{A}^{\beta}$ on both sides of Eq. (\ref{main2}), we get
\begin{align}
\parallel y_n(\varrho) \parallel_{\beta} &\leq \parallel \mathfrak{T}(\varrho-\varsigma_i) \parallel \parallel h_{i,n}(\varsigma_i,y_n(\varrho_i^-)) \parallel_{\beta}+ \int_{\varsigma_i}^\varrho \parallel \mathcal{A}^{\beta} \mathfrak{T}(\varrho-\varsigma) \parallel \parallel \mathfrak{F}_n(\varsigma, (y_{n})_{\varsigma}) \parallel d\varsigma \nonumber \\
& \leq M M_h + M_{\beta} \frac{\varrho_{i+1}^{1-\beta}}{1-\beta}L(R) \nonumber \\
& = M_i \ \leq M^{\prime}.
\end{align}
Thus, the desired result is obtained.
\end{proof}
\section{Convergence of approximate solutions}\label{Section 4}
\begin{theorem}\label{theorem3.4}
Let $\textbf{(A1)}-\textbf{(A4)}$ hold and $\chi(\varrho) \in D(\mathcal{A})$ $\forall \varrho \in [-\tau,0]$. Then $\{y_n\} \subset \mathfrak{PC}_{T}^{\alpha}$ forms a Cauchy sequence and consequently $y_n \rightarrow y$ in $\mathfrak{PC}_{T}^{\alpha}$.
\end{theorem}
\begin{proof} For $0 < \alpha < \beta < 1$ and $n \geq m$ where $m,n \in \mathbb{N}$, we have
\begin{align}\label{3.18}
\parallel \mathfrak{F}_n(\varsigma,(y_{n})_{\varsigma}) - \mathfrak{F}_m(\varsigma,(y_{m})_{\varsigma}) \parallel & \leq \ \parallel \mathfrak{F}_n(\varsigma,(y_{n})_{\varsigma}) - \mathfrak{F}_n(\varsigma,(y_{m})_{\varsigma}) \parallel + \parallel \mathfrak{F}_n(\varsigma,(y_{m})_{\varsigma}) \nonumber \\
& \quad - \mathfrak{F}_m(\varsigma,(y_{m})_{\varsigma}) \parallel \nonumber \\
& \leq L_{\mathfrak{F}} \bigg[\parallel (y_{n})_{\varsigma} - (y_{m})_{\varsigma} \parallel_{0,\alpha} + \parallel \mathcal{A}^{\alpha-\beta} (\mathcal{P}^n-\mathcal{P}^m)\\
& \quad \times \mathcal{A}^{\beta}(y_{m})_{\varsigma} \parallel_{0} \bigg]. \nonumber
\end{align}
The aforesaid inequality is calculated as follows.\\
Let $n > m$ then $\mathcal{H}_n \supset \mathcal{H}_m$ and $\mathcal{H}_n^{\perp} \subset \mathcal{H}_m^{\perp}$, where  $\mathcal{H}_m^{\perp}$  and $\mathcal{H}_n^{\perp}$ be the orthogonal complement of $\mathcal{H}_m$ and $\mathcal{H}_n$ respectively, for all $n,m \in \mathbb{N}_{0}$. Then $\mathcal{H}$ can be written as $\mathcal{H} = \mathcal{H}_m \oplus \mathcal{H}_m^{\perp} = \mathcal{H}_n \oplus \mathcal{H}_n^{\perp}$. Any arbitrary $z \in \mathcal{H}$ has a unique representation $z = z_m + x_m$, where $z_m \in \mathcal{H}_m$ and $x_m \in \mathcal{H}_m^{\perp}$. Then $\mathcal{P}^{m}z = z_m \in \mathcal{H}_m$. Since $x_m \in \mathcal{H}_m^{\perp}$ therefore $x_m = \sum_{i = m+1}^n a_i \psi_{i} + x_m^{\prime},$ where $x_m^{\prime} \in \mathcal{H}_n^{\perp}$.\\
 Let $ \sum_{i = m+1}^n a_i \psi_i = z_m^{\prime}$.\\
Therefore
\begin{equation*}
\mathcal{P}^{n}z - \mathcal{P}^{m}z = z_m^{\prime} = \sum_{i = m+1}^n a_i \psi_i.
\end{equation*}
If $z = \sum_{i = 1}^{\infty} a_i \psi_i$ then $\parallel z \parallel^2 = \sum_{i = 1}^{\infty} |a_i|^2$.\\
Since $\mathcal{A}^{\alpha - \beta} \psi_i = \lambda_i^{\alpha-\beta} \psi_i$, therefore we have
\begin{align}
\parallel \mathcal{A}^{\alpha-\beta}(\mathcal{P}^n - \mathcal{P}^m)z \parallel^2
& \leq \frac{1}{\lambda_{m}^{2(\beta-\alpha)}} \parallel z \parallel^2. \nonumber
\end{align}
Therefore, after performing the above calculations, we reach at the following conclusion:
\begin{align}
\parallel \mathcal{A}^{\alpha - \beta}(\mathcal{P}^n - \mathcal{P}^m) \mathcal{A}^{\beta}y_m(\varsigma) \parallel &\leq \frac{1}{\lambda_{m}^{\beta-\alpha}} \parallel \mathcal{A}^{\beta}y_m(\varsigma) \parallel \nonumber \\
& \leq \frac{1}{\lambda_m^{\beta-\alpha}}M^{\prime}. \nonumber
\end{align}
Therefore, the inequality (\ref{3.18}) becomes
\begin{equation}\label{3.19}
\parallel \mathfrak{F}_n(\varsigma,(y_{n})_{\varsigma}) - \mathfrak{F}_m(\varsigma,(y_m)_{\varsigma}) \parallel \leq L_{\mathfrak{F}}\bigg[ \parallel (y_n)_{\varsigma}- (y_{m})_{\varsigma} \parallel_{0,\alpha} + \frac{1}{\lambda_m^{\beta-\alpha}} M^{\prime}\bigg].
\end{equation}
 For $\varrho \in [-\tau,0]$, we have
\begin{equation}
\lVert y_n(\varrho) - y_m(\varrho) \rVert_{\alpha} = 0.
\end{equation}
Now, for any $\varrho \in (0,\varrho_1],$ we have
\begin{align}
\parallel y_n(\varrho) - y_m(\varrho) \parallel_{\alpha} & \leq \int_0^\varrho \parallel \mathcal{A}^{\alpha} \mathfrak{T}(\varrho-\varsigma) \parallel \parallel \mathfrak{F}_n(\varsigma,(y_{n})_{\varsigma}) - \mathfrak{F}_m(\varsigma,(y_{m})_{\varsigma}) \parallel d\varsigma. \nonumber
\end{align}
In the preceding inequality, we employ the inequality (\ref{3.19}) and get
\begin{align*}
\parallel y_n(\varrho) - y_m(\varrho) \parallel_{\alpha} & \leq \int_0^\varrho M_{\alpha} (\varrho-\varsigma)^{-\alpha}L_{\mathfrak{F}}\bigg[ \parallel (y_n)_{\varsigma}- (y_{m})_{\varsigma} \parallel_{0,\alpha} + \frac{1}{\lambda_m^{\beta-\alpha}} M^{\prime}\bigg]  \nonumber \\
&\leq  \frac{1}{\lambda_m^{\beta-\alpha}}C_1 + M_{\alpha}L_{\mathfrak{F}} \int_0^\varrho (\varrho-\varsigma)^{-\alpha} \parallel y_n - y_m \parallel_{\varsigma,\alpha} d\varsigma
\end{align*}
where, $C_1 = M^{\prime}M_{\alpha} \frac{\varrho_{1}^{1-\alpha}}{1-\alpha}L_{\mathfrak{F}}.$
\vspace{2mm}
\\
Hence,
\begin{align}\label{3.20}
\lVert y_n - y_m \rVert_{\varrho_{1}, \alpha} & \leq  \frac{1}{\lambda_m^{\beta-\alpha}}C_1 + M_{\alpha}L_{\mathfrak{F}} \int_0^\varrho (\varrho-\varsigma)^{-\alpha} \parallel y_n - y_m \parallel_{\varsigma,\alpha} d\varsigma
\end{align}
Now, for any $\varrho \in (\varrho_i,\varsigma_i], \ i = 1,2,\cdots,q$, we have
\begin{align}
\parallel y_n(\varrho) - y_m(\varrho) \parallel_{\alpha} &= \ \parallel h_{i,n}(\varrho,y_n(\varrho_i^-))-h_{i,m}(\varrho,y_m(\varrho_i^-)) \parallel_{\alpha} \nonumber \\
& \leq L_{h_i}\bigg[ \parallel y_n(\varrho_i^-) - y_m(\varrho_i^-) \parallel_{\alpha} + \frac{1}{\lambda_m^{\beta-\alpha}}M^{\prime}\bigg] \nonumber .
\end{align}
Hence,
\begin{equation}\label{3.21}
\parallel y_n - y_m \parallel_{\varsigma_{i},\alpha} \  \leq \frac{1}{\lambda_m^{\beta-\alpha}}  \frac{M^{\prime}L_{h_i}}{(1-L_{h_i})}.
\end{equation}
Similarly, for any $\varrho \in (\varsigma_{i},\varrho_{i+1}], \ i = 1,2,\cdots,q$, we have
\begin{align*}
\parallel y_n(\varrho) - y_m(\varrho) \parallel_{\alpha}  & \leq \parallel \mathfrak{T}(\varrho-\varsigma_i) \parallel \parallel h_{i,n}(\varsigma_i,y_n(\varrho_i^-)) - h_{i,m}(\varsigma_i,y_m(\varrho_i^-)) \parallel_{\alpha} \nonumber\\
& \quad + \int_{\varsigma_i}^\varrho \parallel \mathcal{A}^{\alpha}\mathfrak{T}(\varrho-\varsigma) \parallel \parallel \mathfrak{F}_n(\varsigma,(y_n)_{\varsigma}) - \mathfrak{F}_{m}(\varsigma,(y_m)_{\varsigma}) \parallel d\varsigma \nonumber \\
& \leq ML_{h_i} \bigg[ \parallel y_n(\varrho_i^-) - y_m(\varrho_i^-) \parallel_{\alpha} + \frac{1}{\lambda_m^{\beta-\alpha}}M^{\prime} \bigg] \\ &\quad+ \int_{\varsigma_i}^\varrho
M_{\alpha} (\varrho-\varsigma)^{-\alpha}L_{\mathfrak{F}}\bigg[ \parallel (y_n)_{\varsigma}- (y_{m})_{\varsigma} \parallel_{0,\alpha} + \frac{1}{\lambda_m^{\beta-\alpha}} M^{\prime}\bigg] \nonumber.
\end{align*}
Hence,
\begin{equation}\label{3.22}
\parallel y_n - y_m \parallel_{\varrho_{i+1},\alpha} \ \leq \ \frac{C_2}{\lambda_m^{\beta-\alpha}} + \frac{M_{\alpha}L_{\mathfrak{F}}}{1-ML_{h_i}} \int_{\varsigma_i}^\varrho (\varrho-\varsigma)^{-\alpha} \parallel y_n - y_m \parallel_{\varsigma,\alpha}d\varsigma
\end{equation}
where, $C_2 = \frac{1}{1-ML_{h_i}}MM^{\prime}L_{h_i} + \frac{1}{1-ML_{h_i}}M^{\prime}M_{\alpha}L_{\mathfrak{F}}\frac{T^{1-\alpha}}{1-\alpha}
$.
\vspace{2mm}
\\
By employing Gronwall's inequality and subsequently considering the limit $m \rightarrow \infty$ in the inequalities (\ref{3.20}), (\ref{3.21}) and (\ref{3.22}), we achieve the required result.
\end{proof}
Next, we shall prove the following convergence result by using the Theorem \ref{thrm1} and Theorem \ref{theorem3.4}.
\begin{theorem}
Let $ \chi(\varrho) \in D(\mathcal{A} ) \ \forall \ \varrho \in [-\tau,0]$ and assumptions $\textbf{(A1)}-\textbf{(A4)}$ hold. Then there exist functions $y \in \mathfrak{PC}_{T}^{\alpha}$ and $y_n \in \mathfrak{PC}_{T}^{\alpha}$ satisfying
\begin{equation*}
y_{n}(\varrho)=
\begin{cases}
\ \mathfrak{T}(\varrho)\chi(0) + \int_{0}^{\varrho} \mathfrak{T}(\varrho-\varsigma)\mathfrak{F}_{n}(\varsigma,(y_{n})_{\varsigma})d\varsigma, \quad \varrho \in [0,\varrho_{1}], \\
\ h_{i,n}(\varrho,y_n(\varrho_{i}^{-})), \quad \varrho \in (\varrho_{i},\varsigma_{i}], \quad i=1,2,\cdots,q, \\
\ \mathfrak{T}(\varrho-\varsigma_{i})h_{i,n}(\varsigma_{i},y_n(\varrho_{i}^{-}))+ \int_{\varsigma_{i}}^{\varrho} \mathfrak{T}(\varrho-\varsigma)\mathfrak{F}_{n}(\varsigma,(y_{n})_{\varsigma})d\varsigma, \quad \varrho \in (\varsigma_{i},\varrho_{i+1}], \ i = 1,2,\cdots,q
\end{cases}
\end{equation*}
and
\begin{equation}\label{3.23}
y(\varrho)=
\begin{cases}
\ \mathfrak{T}(\varrho)\chi(0) + \int_{0}^{\varrho} \mathfrak{T}(\varrho-\varsigma)\mathfrak{F}(\varsigma,y_{\varsigma})d\varsigma, \quad \varrho \in [0,\varrho_{1}], \\
\ h_{i}(\varrho,y(\varrho_{i}^{-})), \quad \varrho \in (\varrho_{i},\varsigma_{i}], \quad i=1,2,\cdots,q, \\
\ \mathfrak{T}(\varrho-\varsigma_{i})h_{i}(\varsigma_{i},y(\varrho_{i}^{-}))+ \int_{\varsigma_{i}}^{\varrho} \mathfrak{T}(\varrho-\varsigma)\mathfrak{F}(\varsigma,y_{\varsigma})d\varsigma, \quad \varrho \in (\varsigma_{i},\varrho_{i+1}], \\
\ \qquad \qquad \qquad \qquad \qquad \qquad \qquad\qquad i = 1,2,\cdots,q,
\end{cases}
\end{equation}
such that $y_n$ converges to a unique mild solution y in $\mathfrak{PC}_{T}^{\alpha}$ as $n \rightarrow \infty$.
\end{theorem}
\begin{proof}
From Theorem \ref{theorem3.4}, one can prove the existence of $y$ on $[-\tau, T]$. We only need to show that $y$ is given by Eq. (\ref{3.23}). Further, we have
\begin{align*}
\parallel \mathfrak{F}(\varrho,\mathcal{P}^n(y_n)_{\varrho}) -\mathfrak{F}(\varrho,y_{\varrho}) \parallel &\leq L_{\mathfrak{F}} \parallel \mathcal{P}^n(y_{n})_{\varrho}-y_{\varrho} \parallel_{0,\alpha} \\
&\leq L_{\mathfrak{F}}[\parallel (\mathcal{P}^n-I)(y_n)_{\varrho} \parallel_{0,\alpha}+\parallel y_n-y \parallel_{T,\alpha}].
\end{align*}
Hence, $\parallel \mathfrak{F}(\varrho,\mathcal{P}^n(y_n)_{\varrho})-\mathfrak{F}(\varrho,y_{\varrho}) \parallel \rightarrow 0$ as $n \rightarrow \infty.$\\
Also
\begin{align*}
\parallel h_{i}(\varrho,\mathcal{P}^ny_n(\varrho_i^-))-h_{i}(\varrho,y(\varrho_i^-)) \parallel_{\alpha} &\leq L_{h_i}[\parallel \mathcal{P}^ny_n(\varrho_i^-)-y(\varrho_i^-) \parallel_{\alpha}]\\
&\leq L_{h_i}[\parallel (\mathcal{P}^n-I)y_n(\varrho_i^-) \parallel_{\alpha}+ \parallel y_n(\varrho_i^-)-y(\varrho_i^-) \parallel_{\alpha}].
\end{align*}
Hence, $\parallel h_{i}(\varrho,\mathcal{P}^ny_n(\varrho_i^-)) - h_{i}(\varrho,y(\varrho_i^-)) \parallel_{\alpha} \rightarrow 0$ as $n \rightarrow \infty$.\\
Thus, we have
\begin{align*}
\int_0^\varrho \parallel \mathcal{A}^{\alpha} \mathfrak{T}(\varrho-\varsigma)  \parallel \parallel \mathfrak{F}(\varsigma,\mathcal{P}^n(y_n)_{\varsigma})-\mathfrak{F}(\varsigma,y_{\varsigma}) \parallel d\varsigma \rightarrow 0 \quad \text{when} \quad n \rightarrow \infty
\end{align*}
and
\begin{align*}
\parallel \mathfrak{T}(\varrho-\varsigma_i)\parallel & \parallel h_{i}(\varsigma_i,\mathcal{P}^ny_n(\varrho_i^-))-h_i(\varsigma_i,y(\varrho_i^-)) \parallel_{\alpha} \\
&+ \int_{\varsigma_i}^\varrho \parallel \mathcal{A}^{\alpha} \mathfrak{T}(\varrho-\varsigma) \parallel
 \parallel \mathfrak{F}_n(\varsigma,\mathcal{P}^n(y_n)_{\varsigma})-\mathfrak{F}(\varsigma,y_{\varsigma}) \parallel d\varsigma \rightarrow 0 \quad \text{when} \quad n \rightarrow \infty.
\end{align*}
Consequent to the aforementioned inequalities, we can conclude that $y$ satisfies Eq. (\ref{3.23}). 
\end{proof}
\section{Faedo-Galerkin approximations}\label{Section 5}
This segment pertains to the existence of Faedo-Galerkin approximate solutions of system (\ref{main1}) and some convergence results.\\
As a result of earlier section, there exists a unique $y \in \mathfrak{PC}_{T}^{\alpha}$ for any $-\tau \leq \varrho \leq T$, satisfying the integral equation
\begin{equation}\label{4.1}
y(\varrho)=
\begin{cases}
\ \chi(\varrho), \quad \varrho \in [-\tau, 0],\\
\ \mathfrak{T}(\varrho)\chi(0) + \int_{0}^{\varrho} \mathfrak{T}(\varrho-\varsigma)\mathfrak{F}(\varsigma,y_{\varsigma})d\varsigma, \quad \varrho \in [0,\varrho_{1}], \\
\ h_{i}(\varrho,y(\varrho_{i}^{-})), \quad \varrho \in (\varrho_{i},\varsigma_{i}], \quad i=1,2,\cdots,q, \\
\ \mathfrak{T}(\varrho-\varsigma_{i})h_{i}(\varsigma_{i},y(\varrho_{i}^{-}))+ \int_{\varsigma_{i}}^{\varrho} \mathfrak{T}(\varrho-\varsigma)\mathfrak{F}(\varsigma,y_{\varsigma})d\varsigma, \ \varrho \in (\varsigma_{i},\varrho_{i+1}], \\
\ \qquad \qquad \qquad \qquad \qquad \qquad \qquad \qquad i = 1,2,\cdots,q.
\end{cases}
\end{equation}
Also, a unique solution $y_n \in \mathfrak{PC}_{T}^{\alpha}$ satisfies the approximate integral equation
\begin{equation}\label{4.2}
y_{n}(\varrho)=
\begin{cases}
\ \chi(\varrho), \quad \varrho \in [-\tau, 0],\\
\ \mathfrak{T}(\varrho)\chi(0) + \int_{0}^{\varrho} \mathfrak{T}(\varrho-\varsigma)\mathfrak{F}_{n}(\varsigma,(y_{n})_{\varsigma})d\varsigma, \quad \varrho \in [0,\varrho_{1}], \\
\ h_{i,n}(\varrho,y_n(\varrho_{i}^{-})), \quad \varrho \in (\varrho_{i},\varsigma_{i}], \quad i=1,2,\cdots,q, \\
\ \mathfrak{T}(\varrho-\varsigma_{i})h_{i,n}(\varsigma_{i},y_n(\varrho_{i}^{-}))+ \int_{\varsigma_{i}}^{\varrho} \mathfrak{T}(\varrho-\varsigma)\mathfrak{F}_{n}(\varsigma,(y_{n})_{\varsigma})d\varsigma, \ \varrho \in (\varsigma_{i},\varrho_{i+1}], \\
\ \qquad \qquad \qquad \qquad \qquad \qquad \qquad \qquad \quad  i = 1,2,\cdots,q.
\end{cases}
\end{equation}
If we project Eq. (\ref{4.2}) onto $\mathcal{H}_n$, we get approximations i.e. $\bar{y}_n =\mathcal{P}^n y_n$ in the finite dimensional space, that are referred to as Faedo-Galerkin approximations satisfying
\begin{equation}\label{4.3}
\bar{y}_n(\varrho)=
\begin{cases}
\ \mathcal{P}^n \chi(\varrho), \quad \varrho \in [-\tau, 0],\\
\ \mathfrak{T}(\varrho)\mathcal{P}^n \chi(0) + \int_{0}^{\varrho} \mathfrak{T}(\varrho-\varsigma)\mathcal{P}^n \mathfrak{F}_{n}(\varsigma,(y_{n})_{\varsigma})d\varsigma, \quad \varrho \in [0,\varrho_{1}], \\
\ \mathcal{P}^n h_{i,n}(\varrho,y_n(\varrho_{i}^{-})), \quad \varrho \in (\varrho_{i},\varsigma_{i}], \quad i=1,2,\cdots,q, \\
\ \mathfrak{T}(\varrho-\varsigma_{i})\mathcal{P}^n h_{i,n}(\varsigma_{i},y_n(\varrho_{i}^{-}))+ \int_{\varsigma_{i}}^{\varrho} \mathfrak{T}(\varrho-\varsigma)\mathcal{P}^n \mathfrak{F}_{n}(\varsigma,(y_{n})_{\varsigma})d\varsigma, \ \varrho \in (\varsigma_{i},\varrho_{i+1}],\\
\  \qquad \qquad \qquad \qquad \qquad \qquad \qquad \qquad \qquad \qquad i = 1,2,\cdots,q,
\end{cases}
\end{equation}
where $\mathfrak{F}_n$ and $h_{i,n}$ are as previously defined.\\
The solution $y$ of (\ref{4.1}) and $\bar{y}_n$ of (\ref{4.3}) are represented by
\begin{equation}
 y(\varrho) = \sum_{i=0}^{\infty} \alpha_i(\varrho) \psi_i, \quad  \alpha_i = \langle  y(\varrho), \psi_i \rangle , \ \forall i=0,1,2,\cdots,
\end{equation}
\begin{equation}
\bar{y}_n(\varrho) = \sum_{i=0}^{n} \alpha_i^n(\varrho) \psi_i, \quad \alpha_i^n(\varrho) = \langle \bar{y}_n(\varrho), \psi_i \rangle, \ \forall i = 0,1,2,\cdots,n.
\end{equation}
\begin{theorem}
Let $\chi(\varrho) \in D(\mathcal{A}) \forall \ \varrho \in [-\tau, 0]$ and all the assumptions $\textbf{(A1)}-\textbf{(A4)}$ hold. Then,
\begin{eqnarray*}
\lim_{m \rightarrow \infty} \sup_{\{n \geq m, -\tau \leq \varrho \leq T\}} \lVert \bar{y}_{n}(\varrho) - \bar{y}_{m}(\varrho) \rVert_{\alpha} = 0.
\end{eqnarray*}
\begin{proof}
For $n \geq m$ and $0 < \alpha < \beta$, we have
\begin{align*}
\lVert \bar{y}_n(\varrho) -\bar{y}_m(\varrho) \rVert_{\alpha} &= \lVert \mathcal{P}^n y_n(\varrho) - \mathcal{P}^m y_m(\varrho) \rVert_{\alpha} \\
& \leq \lVert \mathcal{P}^n (y_n(\varrho) - y_m(\varrho)) \rVert_{\alpha} + \lVert (\mathcal{P}^n - \mathcal{P}^m ) y_m(\varrho) \rVert_{\alpha} \\
& \leq \lVert y_n(\varrho) - y_m(\varrho) \rVert_{\alpha} + \frac{1}{\lambda_m^{\beta-\alpha}} \parallel \mathcal{A}^{\beta} y_{m}(\varrho) \parallel.
\end{align*}
The desired result is obtained by the utilization of Theorem \ref{theorem3.4}.
\end{proof}
\end{theorem}
Next, we get the following outcomes as a result of Theorem \ref{thrm1} and Theorem \ref{theorem3.4}.
\begin{theorem}\label{thrm 5.2}
Let $\chi(\varrho) \in D(\mathcal{A}) \ \forall \ \varrho \in [-\tau,0]$ and assumptions $\textbf{(A1)}-\textbf{(A4)}$ hold, then there exist functions $y \in \mathfrak{PC}([-\tau ,T];\mathcal{H})$ and  $\bar{y}_n \in \mathfrak{PC}([-\tau ,T];\mathcal{H}_{n})$ fulfil Eqs.~(\ref{4.1}) and (\ref{4.3}) respectively such that $\bar{y}_n \rightarrow y$ in $\mathfrak{PC}([-\tau, T];\mathcal{H})$ as $n \rightarrow \infty$.
\end{theorem}
\begin{theorem}
Let $ \chi(\varrho) \in D(\mathcal{A}) \forall \ \varrho \in [-\tau, 0]$ and $\textbf{(A1)}-\textbf{(A4)}$ hold. Then, the result is as follows
\begin{equation*}
\lim_{n \rightarrow \infty} \ \sup_{-\tau \leq \varrho \leq T} \bigg( \sum_{j = 0}^n \lambda_{j}^{2 \alpha} | \alpha_{j}(\varrho) - \alpha_j^n(\varrho)|^2 \bigg) = 0.
\end{equation*}
\end{theorem}
\begin{proof}
Since
\begin{align*}
\parallel  y(\varrho)-\bar{y}_n(\varrho) \parallel_{\alpha}^2 &= \ \parallel \mathcal{A}^{\alpha}\big[y(\varrho)-\bar{y}_n(\varrho)\big] \parallel^2 \\
 &= \ \parallel \mathcal{A}^{\alpha}\bigg[ \sum_{j = 0}^{\infty}\big( \alpha_j(\varrho)-\alpha_j^n(\varrho)\big) \psi_{j}\bigg] \parallel^2 \\
&= \sum_{j = 0}^{\infty} \lambda_{j}^{2\alpha}|\alpha_j(\varrho)-\alpha_j^n(\varrho)|^2.
\end{align*}
Therefore,
\begin{align*}
 \sum_{j = 0}^n \lambda_j^{2 \alpha}| \alpha_j(\varrho)-\alpha_j^n(\varrho)|^2 \leq \ \parallel \big[y(\varrho)-\bar{y}_n(\varrho) \big] \parallel^2_{\alpha}.
\end{align*}
The required result follows from a consequence of Theorem \ref{thrm 5.2}.
\end{proof}
\section{Example}\label{Section 6}
In this segment, an illustrative example is presented to showcase the feasibility of our abstract findings. In this application, let $\mathcal{H} = L^2(0,1)$.\\
\begin{align}\label{5.1}
\frac{\partial w}{\partial \varrho}(\varrho, \xi)  &= \frac{\partial^2 w}{\partial \xi^2}(\varrho, \xi)+ G(\varrho,w(\varrho-\tau, \xi)), \quad (\varrho, \xi) \in \cup_{i = 0}^q (\varsigma_i, \varrho_{i+1}] \times (0,1), \nonumber \\
w(\varrho, \xi) &= \frac{1}{3} \text{sin}(i\varrho)w(\varrho_i^-, \xi) + \frac{w(\varrho_i^-, \xi)}{2+w(\varrho_i^-, \xi)}, \quad (\varrho,\xi) \in (\varrho_i,\varsigma_i] \times (0,1),   \\
w(\varrho,0) &= w(\varrho,1) = 0, \quad \varrho \in [0,T], \quad 0 < T < \infty, \nonumber \\
w(\varrho, \xi) &= \chi(\varrho,\xi), \quad (\varrho,\xi) \in [-\tau, 0] \times (0,1), \nonumber
\end{align}
where $\chi$ is locally H$\ddot{\text{o}}$lder continuous on $[-\tau, 0]$ satisfies $\chi(0, 0) = \chi(0,1) = 0$ and $0 = \varsigma_0 < \varrho_1 < \varsigma_1 < \cdots  < \varrho_i < \varsigma_i < \varrho_{i+1} = T$ are some real numbers. The function $G : [0,T] \times \mathcal{H} \rightarrow \mathcal{H}$ is given as
\begin{equation*}
G(\varrho,w(\varrho-\tau,\xi))= \beta(\varrho) \int_0^1 f(\xi-s)w(\varrho-\tau, s)ds,
\end{equation*}
where $\beta \in C([0,T];\mathbb{R})$, $f \in C(\mathbb{R}^2;\mathbb{R})$.\\
Let the operator $\mathcal{A} : D(\mathcal{A}) \subset \mathcal{H} \rightarrow \mathcal{H}$ be defined as:
\begin{equation}\label{5.2}
\mathcal{A}w = -\frac{\partial^2 w}{\partial \xi^2} = -w^{\prime \prime} \ \text{with} \ D(\mathcal{A}) = \{ w \in \mathcal{H}_0^1(0,1) \cap   \mathcal{H}^2(0,1) : w^{\prime \prime} \in \mathcal{H} \}.
\end{equation}
Here, the operator $\mathcal{A}$ is self-adjoint with compact resolvent and $-\mathcal{A}$ is the infinitesimal generator of an analytic $C_0$ semigroup $\{\mathfrak{T}(\varrho)\}_{\varrho \geq 0}$.  Now, let $\alpha = \frac{1}{2}$, $D(\mathcal{A}^\frac{1}{2})$ is a Banach space equipped with the norm
\begin{equation}
\parallel y \parallel_{\frac{1}{2}} = \ \parallel \mathcal{A}^{\frac{1}{2}}y \parallel, \quad y \ \in D(\mathcal{A}^\frac{1}{2}),
\end{equation}
and we enlighten this space by $\mathcal{H}_\frac{1}{2} = \mathcal{H}_0^1(0,1).$
For $y \in D(\mathcal{A})$ and $\lambda \in \mathbb{R}$ with $\mathcal{A}y = -y^{\prime \prime} = \lambda y,$ we have $\langle \mathcal{A}y,y \rangle = \langle \lambda y, y \rangle$; that is
\begin{equation*}
\langle -y^{\prime \prime},y \rangle = |y^{\prime}|^2_{L^2} = \lambda |y|^2_{L^2}.
\end{equation*}
Therefore, $\lambda > 0$. Now, the above homogeneous differential equation possess the following general solution
\begin{equation*}
y(\xi) = C \ \text{cos}(\sqrt{\lambda}\xi) + D \ \text{sin}(\sqrt{\lambda}\xi),
\end{equation*}
and the conditions $y(0) = y(1) = 0$ implies that $C = 0$ and spectrum $\lambda = \lambda_n = n^2 \pi^2$, $n \in \mathbb{N}$. Therefore, the corresponding solutions (orthonormal eigenfunctions) are provided by
\begin{equation*}
y_n(\xi) = \sqrt{2} \ \text{sin}(n \pi \xi), \quad n \in \mathbb{N}.
\end{equation*}
Also, there exists a sequence of real numbers $c_{n}$ such that
\begin{equation*}
y(\xi) = \sum_{n = 1}^{\infty} c_n y_n(\xi), \quad \text{ for any} \quad y \in D(\mathcal{A})
\end{equation*}
with
\begin{equation*}
\sum_{n = 1}^{\infty} c_{n}^2 < \infty \quad \text{and} \quad \sum_{n = 1}^{\infty} \lambda_n^2 c_{n}^2 < \infty.
\end{equation*}
For $y \in D(\mathcal{A}^\frac{1}{2})$, we have
\begin{equation*}
\mathcal{A}^{\frac{1}{2}} y(\xi) = \lambda^{\frac{1}{2}} y(\xi) = \sum_{n = 1}^{\infty} \lambda^{\frac{1}{2}} \gamma_n y_n(\xi)
\end{equation*}
that is $\sum_{n = 1}^{\infty} \lambda_n c_{n}^2 < \infty$.\\
Now, we can transform the system (\ref{5.1}) into the following abstract system in $\mathcal{H} = L^2(0,1)$:
\begin{align*}
\frac{d}{d\varrho}y(\varrho) &= -\mathcal{A} y(\varrho) + \mathfrak{F}(\varrho, y_{\varrho}), \quad \varrho \in (\varsigma_i, \varrho_{i+1}], \ i = 0,1,2, \cdots,q, \nonumber \\
y(\varrho) &= h_i(\varrho, y(\varrho_i^-)), \quad \varrho \in (\varrho_i, \varsigma_i], \ i =1,2,\cdots, q, \nonumber \\
y(0) &= y_0,
\end{align*}
where $y(\varrho) = w(\varrho, \cdot)$ that is $y(\varrho)(\xi) = w(\varrho,\xi)$, $(\varrho,\xi) \in [0,T] \times (0,1),$ $y_{\varrho}(\theta)(\xi) = y(\varrho+\theta, \xi)$, $\varrho \in [0,T], \ \xi \in (0,1), \ \theta \in [-\tau, 0]$. Equation (\ref{5.2}) defines the operator $\mathcal{A}$.\\
We define the operators by\\
\vspace{2mm}
(i) $\mathfrak{F} : [0,T] \times \mathfrak{PC}_{0}^{\frac{1}{2}} \rightarrow \mathcal{H}$ is given by
\begin{align*}
\mathfrak{F}(\varrho,\psi)(\xi) = \beta(\varrho) \int_0^1 f(\xi-s) \psi(-\tau) ds,
\end{align*}
(ii) $h_{i} : (\varrho_i,\varsigma_i] \times \mathcal{H}_{\frac{1}{2}} \rightarrow \mathcal{H}_{\frac{1}{2}}$ is given by
\begin{equation*}
h_i(\varrho,y(\varrho_i^-))(\xi) = \frac{1}{3} \text{sin}(i\varrho)y(\varrho_i^-) + \frac{y(\varrho_i^-)}{2+y(\varrho_i^-)}.
\end{equation*}
For each $\varrho \in [0,T]$ and ${\psi}_{1}, {\psi}_{2} \in \mathfrak{PC}_{0}^{\frac{1}{2}}$, we have
\begin{equation*}
\parallel \mathfrak{F}(\varrho, {\psi}_1) - \mathfrak{F}(\varrho,{\psi}_2) \parallel_{L^2(0,1)} \ \leq \ \parallel \beta \parallel_{C([0,T];\mathbb{R})} \Theta \parallel \mathcal{A}^{\frac{-1}{2}} \parallel \parallel {\psi}_1 - {\psi}_2 \parallel_{0,\frac{1}{2}}
\end{equation*}
where,
\begin{equation*}
\Theta = \bigg(\int_0^1 \int_0^1 f^2(\xi-s) ds dy\bigg)^{\frac{1}{2}} < \infty.
\end{equation*}
And
\begin{equation*}
\parallel h_i(\varrho,y_1)-h_i(\varrho,y_2) \parallel_\frac{1}{2} \ \leq \frac{1}{3} \parallel y_1-y_2 \parallel_\frac{1}{2} + \frac{1}{2} \parallel y_1-y_2 \parallel_\frac{1}{2}.
\end{equation*}
It is easy to see that all the assumptions are verified. For the particular case, we can take
\begin{equation*}
\mathfrak{F}(\varrho,y_{\varrho}) = \int_0^1 K(\varrho-s)g(y_{\varrho})ds
\end{equation*}
where K and g are Lipschitz continuous on $\mathcal{H}$. For instance, $K(\varrho-s) = \beta \sin(\varrho-s)$ and $g(y) = \text{arctan}(y)$, where $\beta$ is a constant.
\section{Conclusion}
This paper has successfully addressed the existence, uniqueness, and convergence results for approximate solutions to a non-instantaneous impulsive retarded differential equation in a separable Hilbert space. The existence and uniqueness of solutions for each approximate integral equation is proven through the use of analytic semigroup theory and the Banach fixed point technique. The article further demonstrates that these approximations form a Cauchy sequence under a suitable norm.  As a consequence, the solution to the original problem emerges as the limit of this sequence. Further, we introduced Faedo-Galerkin approximations and their convergence. An application to partial differential equations is presented in the last portion to exemplify abstract results.

\section*{Declaration}
\textbf{Conflict of interest} The authors declare that they have no conflict of interest.\vspace{5mm}\\
\textbf{Funding} Not applicable.

\end{document}